\documentclass[12pt]{amsart}

\usepackage{amsmath,amssymb,amsfonts,amsxtra,hyperref,fullpage,xypic,centernot,setspace}
\usepackage{soul}
\setuldepth{zarikian}

\nocite{*}

\theoremstyle{plain}
\newtheorem*{theorem*}{Theorem}
\newtheorem{theorem}{Theorem}[subsection]
\newtheorem{lemma}[theorem]{Lemma}

\newtheorem{corollary}[theorem]{Corollary}
\newtheorem{question}[theorem]{Question}
\newtheorem{example}[theorem]{Example}
\theoremstyle{definition}
\newtheorem{remark}[theorem]{Remark}

\DeclareMathOperator{\bbC}{\mathbb{C}}
\DeclareMathOperator{\bbE}{\mathbb{E}}
\DeclareMathOperator{\bbF}{\mathbb{F}}
\DeclareMathOperator{\bbK}{\mathbb{K}}
\DeclareMathOperator{\bbM}{\mathbb{M}}
\DeclareMathOperator{\bbN}{\mathbb{N}}

\DeclareMathOperator{\bbR}{\mathbb{R}}

\DeclareMathOperator{\bbZ}{\mathbb{Z}}
\DeclareMathOperator{\A}{\mathcal{A}}
\DeclareMathOperator{\Ad}{\operatorname{Ad}}
\DeclareMathOperator{\Aut}{\operatorname{Aut}}
\DeclareMathOperator{\B}{\mathcal{B}}
\DeclareMathOperator{\C}{\mathcal{C}}
\DeclareMathOperator{\D}{\mathcal{D}}

\DeclareMathOperator{\ess}{\operatorname{ess}}

\renewcommand{\H}{\mathcal{H}}

\DeclareMathOperator{\id}{\operatorname{id}}
\DeclareMathOperator{\J}{\mathcal{J}}

\DeclareMathOperator{\LIM}{\operatorname{LIM}}
\DeclareMathOperator{\M}{\mathcal{M}}

\let\slasho=\o
\renewcommand{\o}{\overline}

\DeclareMathOperator{\Proj}{\operatorname{Proj}}

\renewcommand{\span}{\operatorname{span}}
\DeclareMathOperator{\U}{\mathcal{U}}

\begin{document}

\nocite{*}

\title{Unique Pseudo-Expectations for\\ Hereditarily Essential $C^*$-Inclusions}
\author{Vrej Zarikian} \thanks{The author was partially supported by a AMS-Simons Research Enhancement Grant for PUI (Primarily Undergraduate Institution) Faculty.}
\address{U. S. Naval Academy, Annapolis, MD 21402}
\email{zarikian@usna.edu}
\dedicatory{Dedicated to the memory of Gary Weiss}
\date{\today}
\begin{abstract}
The $C^*$-inclusion $\A \subseteq \B$ is said to be hereditarily essential if for every intermediate $C^*$-algebra $\A \subseteq \C \subseteq \B$ and every non-zero ideal $\{0\} \neq \J \lhd \C$, we have that $\J \cap \A \neq \{0\}$. That is, $\A$ detects ideals in every intermediate $C^*$-algebra $\A \subseteq \C \subseteq \B$. By a result of Pitts and the author, the $C^*$-inclusion $\A \subseteq \B$ is hereditarily essential if and only if every pseudo-expectation $\theta:\B \to I(\A)$ for $\A \subseteq \B$ is faithful. A decade-old open question asks whether hereditarily essential $C^*$-inclusions must have \ul{unique} pseudo-expectations? In this note, we answer the question affirmatively for some important classes of $C^*$-inclusions, in particular those of the form $\A \rtimes_{\alpha,r}^\sigma N \subseteq \A \rtimes_{\alpha,r}^\sigma G$, for a discrete twisted $C^*$-dynamical system $(\A,G,\alpha,\sigma)$ and a normal subgroup $N \lhd G$. On the other hand, we settle the general question negatively by exhibiting $C^*$-irreducible inclusions of the form $C_r^*(G) \subseteq C(X) \rtimes_{\alpha,r} G$ with multiple conditional expectations. Our results leave open the possibility that the question has a positive answer for \ul{regular} hereditarily essential $C^*$-inclusions.
\end{abstract}
\maketitle


\section{Introduction}

\subsection{$C^*$-Inclusions}

By a \emph{$C^*$-inclusion} we mean an inclusion $\A \subseteq \B$ of $C^*$-algebras. We say that the $C^*$-inclusion is \emph{unital} if $\A$ is unital and $1_{\A}$ is a unit for $\B$. We say that the $C^*$-inclusion is \emph{approximately unital} if there exists an approximate unit $\{e_\lambda\}$ for $\A$ which is also an approximate unit for $\B$. The study of ``largeness'' conditions for $C^*$-inclusions (i.e., properties of $\A \subseteq \B$ that indicate $\A$ is a ``large'' subalgebra of $\B$) has attracted considerable attention recently, and this work contributes to that program.\\

\subsection{Hereditarily Essential $C^*$-Inclusions} \label{hered}

We say that the $C^*$-inclusion $\A \subseteq \B$ is \emph{essential} if for every non-zero ideal $\{0\} \neq \J \lhd \B$, we have that $\J \cap \A \neq \{0\}$. In other words, $\A$ ``detects ideals'' in $\B$. Intuitively, if $\A \subseteq \B$ is essential, then $\A$ should be a ``large'' subalgebra of $\B$, because no non-trivial ideal of $\B$ can avoid it. Unfortunately, this intuition can fall apart when $\B$ has few ideals. Indeed, if $\B$ is a simple $C^*$-algebra, then $\A \subseteq \B$ is essential for every non-zero subalgebra $\A$, no matter how small. So a stronger condition is needed.\\

We say that the $C^*$-inclusion $\A \subseteq \B$ is \emph{hereditarily essential} if for every intermediate $C^*$-algebra $\A \subseteq \C \subseteq \B$, the $C^*$-inclusion $\A \subseteq \C$ is essential. In other words, $\A$ ``detects ideals'' in every intermediate $C^*$-algebra $\A \subseteq \C \subseteq \B$. This condition is studied explicitly in \cite{PittsZarikian2015} and \cite{KwasniewskiMeyer2022}. It is also studied implicitly in \cite{Rordam2023}, a paper devoted to the systematic investigation of \emph{$C^*$-irreducible inclusions}, unital $C^*$-inclusions $\A \subseteq \B$ for which every intermediate $C^*$-algebra $\A \subseteq \C \subseteq \B$ is simple. It is easy to see that the $C^*$-irreducible inclusions are precisely the unital hereditarily essential $C^*$-inclusions $\A \subseteq \B$ for which $\A$ is simple.

\subsection{Pseudo-Expectations} \label{PsExp}

A \emph{conditional expectation} for the $C^*$-inclusion $\A \subseteq \B$ is a c.c.p. (contractive completely positive) map $E:\B \to \A$ such that $E(a)=a$ for all $a \in \A$. By Choi's Lemma, $E$ is an $\A$-bimodule map, meaning that $E(a_1ba_2)=a_1E(b)a_2$ for all $b \in \B$ and $a_1, a_2 \in \A$ (see \cite[Section 1.3.12]{BlecherLeMerdy2004} for the unital case).\\

In order to deal with the fact that many $C^*$-inclusions have no conditional expectations at all, Pitts introduced \emph{pseudo-expectations} \cite[Definition 1.3]{Pitts2017}. A pseudo-expectation for the $C^*$-inclusion $\A \subseteq \B$ is a c.c.p. map $\theta:\B \to I(\A)$ such that $\theta(a)=a$ for all $a \in \A$. Here $I(\A)$ is the \emph{injective envelope} of $\A$, i.e., the smallest injective operator space containing $\A$ \cite{Hamana1979}. (In fact, $I(\A)$ is a unital $C^*$-algebra containing $\A$ as a $C^*$-subalgebra.) Pseudo-expectations generalize conditional expectations, but always exist because of injectivity. If $\theta:\B \to I(\A)$ is a pseudo-expectation for $\A \subseteq \B$, then (just like conditional expectations) $\theta$ is an $\A$-bimodule map. A pseudo-expectation is said to be \emph{faithful} if $\theta(b^*b)=0$ implies $b=0$.\\

An intimate connection between pseudo-expectations and hereditarily essential $C^*$-inclusions was established in \cite[Theorem 3.5]{PittsZarikian2015}, at least for unital $C^*$-inclusions. We state the result for arbitrary $C^*$-inclusions, providing the nearly-identical proof for the reader's convenience.

\begin{theorem}[\cite{PittsZarikian2015}] \label{equivalence}
The $C^*$-inclusion $\A \subseteq \B$ is hereditarily essential if and only if every pseudo-expectation is faithful.
\end{theorem}

\begin{proof}
($\Rightarrow$) Suppose $\A \subseteq \B$ is hereditarily essential. Let $\theta:\B \to I(\A)$ be a c.c.p. map such that $\theta|_{\A}=\id$. Assume that $\theta(b^*b)=0$, where $b \in \B$. Let $\C=C^*(\A,|b|)$, so that $\A \subseteq \C \subseteq \B$, and let $\J \lhd \C$ be the ideal generated by $|b|$. It is easy to see that
\[
    \J = \o{\span}(\{c|b|: c \in \C\} \cup \{c|b|a: c \in \C, ~ a \in \A\}).
\]
Now for all $c \in \C$, we have that
\[
    \theta(c|b|)^*\theta(c|b|) \leq \theta(|b|c^*c|b|) \leq \|c\|^2\theta(b^*b) = 0.
\]
Likewise, for all $c \in \C$ and $a \in \A$, we have that
\[
    \theta(c|b|a)^*\theta(c|b|a) \leq \theta(a^*|b|c^*c|b|a) \leq \|c\|^2\theta(a^*b^*ba) = \|c\|^2a^*\theta(b^*b)a = 0.
\]
It follows that $\theta(j)=0$ for all $j \in \J$. So if $h \in \J \cap \A$, then
\[
    h=\theta(h)=0.
\]
That is, $\J \cap \A = \{0\}$. Since $\A \subseteq \C$ is essential, $\J=\{0\}$, which implies $|b|=0$, which in turn implies $b=0$. So $\theta$ is faithful.

($\Leftarrow$) Conversely, suppose every pseudo-expectation for $\A \subseteq \B$ is faithful. Let $\A \subseteq \C \subseteq \B$ and $\J \lhd \C$. Assume that $\J \cap \A = \{0\}$. Then the map $\pi:\A+\J \to \A:a+j \mapsto a$ is a well-defined $*$-homomorphism. By injectivity, there exists a c.c.p. map $\theta:\B \to I(\A)$ such that $\theta|_{\A+\J}=\pi$. Then $\theta$ is a pseudo-expectation for $\A \subseteq \B$, therefore faithful. It follows that $\pi$ is faithful, which implies $\J=\{0\}$. So $\A \subseteq \C$ is essential.
\end{proof}

\subsection{The Main Question} \label{question}

In light of Theorem \ref{equivalence} above, the condition ``every pseudo-expectation is faithful'' has significant structural consequences for a $C^*$-inclusion. An obvious way to establish that \ul{every} pseudo-expectation is faithful is to show that there is a \ul{unique} pseudo-expectation which happens to be faithful. Lacking examples to the contrary, Pitts and the author asked whether this is the only way. Namely, we posed:
 
\begin{question}[\cite{PittsZarikian2015}, Section 7.1, Q6] \label{Q1}
If every pseudo-expectation for a $C^*$-inclusion $\A \subseteq \B$ is faithful (equivalently, if $\A \subseteq \B$ is hereditarily essential), must there be a unique pseudo-expectation?
\end{question}

In \cite{PittsZarikian2015}, Pitts and the author identified two classes of $C^*$-inclusions for which Question \ref{Q1} has an affirmative answer:
\begin{itemize}
\item unital abelian $C^*$-inclusions $C(X) \subseteq C(Y)$ \cite[Corollary 3.22]{PittsZarikian2015};
\item $W^*$-inclusions $\D \subseteq \M$ with $\D$ injective \cite[Proposition 7.5]{PittsZarikian2015}.
\end{itemize}
In \cite{KwasniewskiMeyer2022}, Kwa\'{s}niewski and Meyer utilized their notion of \emph{aperiodicity} (see \cite[Definition 2.3]{KwasniewskiMeyer2022}) to expand the list substantially, adding a wide swath of ``dynamical'' $C^*$-inclusions:
\begin{itemize}
\item $C^*$-inclusions $\A \subseteq \A \rtimes_{\ess} S$, where $\A$ is essentially separable, essentially simple, or essentially Type I, and $\A \rtimes_{\ess} S$ is the \emph{essential crossed product} of $\A$ by the action of an inverse semigroup $S$ \cite[Theorems 7.2 and 7.3]{KwasniewskiMeyer2022}.
\end{itemize}
The purpose of this note is twofold. On the one hand, we add modestly to the classes of $C^*$-inclusions for which Question \ref{Q1} is known to have a positive answer. On the other hand, we settle the general question \ul{negatively}, leading to a reformulation of the question.\\

\subsection{Main Results} \label{main}

If the inverse semigroup $S$ in the aforementioned result of Kwa\'{s}niewski and Meyer is actually a group $G$, then $\A \rtimes_{\ess} S = \A \rtimes_r G$, the reduced crossed product. In that case, we show that Question \ref{Q1} has an affirmative answer with no restrictions on $\A$ (Theorem \ref{crossed product}).\footnote{The unital version of this result was first announced at GPOTS 2018 (Miami, OH), and then again at IWOTA 2019 (Lisbon), but remained unpublished for lack of complementary results.} By a ``stabilization'' trick, we extend this result to \ul{twisted} crossed products (Theorem \ref{twisted}), and by an ``iteration'' trick to twisted crossed products by normal subgroups (Corollary \ref{normal subgroup}).

\begin{theorem*}[main positive result]
Let $(\A,G,\alpha,\sigma)$ be a discrete twisted $C^*$-dynamical system and $N \lhd G$ be a normal subgroup. If every pseudo-expectation for the $C^*$-inclusion $\A \rtimes_{\alpha,r}^\sigma N \subseteq \A \rtimes_{\alpha,r}^\sigma G$ is faithful, then there exists a unique pseudo-expectation.
\end{theorem*}

To be clear, our results are subsumed by the results of Kwa\'{s}niewski and Meyer, except when $\A$ fails to have a reasonably nice essential ideal. That being said, our arguments involve a hands-on analysis of pseudo-expectations, which might be independently interesting.\\

On the other hand, we show that in general Question \ref{Q1} has a negative answer by exhibiting $C^*$-irreducible inclusions with multiple conditional expectations (Theorem \ref{counterexample}). (We recall from Section \ref{hered} that $C^*$-irreducible inclusions are hereditarily essential, and from Section \ref{PsExp} that conditional expectations are pseudo-expectations.)

\begin{theorem*}[main negative result]
There exist $C^*$-irreducible inclusions of the form $C_r^*(G) \subseteq C(X) \rtimes_r G$ with multiple conditional expectations.
\end{theorem*}

Our counterexamples are not \emph{regular}, leading us to revise Question \ref{Q1} accordingly (see Question \ref{Q3} in the concluding remarks).

\section{Positive Results}

\subsection{Crossed Products by Discrete Groups} \label{crossed}

In Theorem \ref{crossed product} below, we show that Question \ref{Q1} has an affirmative answer for $C^*$-inclusions of the form $\A \subseteq \A \rtimes_{\alpha,r} G$, where $(\A,G,\alpha)$ is a discrete $C^*$-dynamical system. In some sense, half the battle has already been fought---uniqueness of pseudo-expectations for such inclusions has been characterized in terms of the dynamics, at least in the unital case \cite[Theorem 3.5]{Zarikian2019a}. The relevant notion is \emph{proper outerness} of the action $\alpha:G \curvearrowright \A$ (see \cite[Section 2.3]{Zarikian2019a}). For general $C^*$-inclusions $\A \subseteq \B$, the correct generalization of proper outerness appears to be \emph{aperiodicity} (see \cite[Definition 2.3]{KwasniewskiMeyer2022}). Aperiodicity implies uniqueness of pseudo-expectations \cite[Theorem 3.6]{KwasniewskiMeyer2022}, and the reverse implication remains a tantalizing open problem.\\

The following dynamical lemma will be needed for the proof. It says that a pointwise inner action of a discrete cyclic group on an $AW^*$-algebra is actually an inner action. The corresponding statement for general $C^*$-algebras is false \cite[Example 8.2.8]{GKPT2018}.

\begin{lemma} \label{inner action}
Let $\A$ be an $AW^*$-algebra, $G = \langle g \rangle$ be a discrete cyclic group, and $\alpha:G \to \Aut(\A)$ be a homomorphism. If $\alpha_g$ is inner, then there exists a unitary representation $u:G \to \A$ such that $\alpha_h = \Ad(u_h)$, $h \in G$.
\end{lemma}

\begin{proof}
By assumption, $\alpha_g = \Ad(v)$ for some $v \in \U(\A)$. It follows that $\alpha_{g^k} = \Ad(v^k)$, $k \in \bbZ$. If $|g| = \infty$, then we may define $u_{g^k} = v^k$, $k \in \bbZ$. Suppose instead that $|g| = n$ for some $n \in \bbN$. Then $\Ad(v^n) = \alpha_{g^n} = \id$, which implies $v^n \in Z(\A)$. Since $\A$ is an $AW^*$-algebra, so is $Z(\A)$ \cite[Section 4, Proposition 8(v)]{Berberian1972}. By \cite[Lemma 6]{Sakai1955}, there exists $t \in Z(\A)_{sa}$ such that $v^n = e^{it}$. Then we may define $u_{g^k} = (e^{-it/n}v)^k$, $k \in \bbZ$.
\end{proof}

Before turning to the proof of the Theorem \ref{crossed product}, we highlight three facts about discrete crossed products which will be used without explicit mention:
\begin{itemize}
\item If $G$ is amenable, then $\A \rtimes_{\alpha,r} G = \A \rtimes_\alpha G$ \cite[Theorem 7.7.7]{Pedersen1979}.
\item If $\A_0 \subseteq \A$ is $\alpha$-invariant, then $\A_0 \rtimes_{\alpha,r} G \subseteq \A \rtimes_{\alpha,r} G$ \cite[Proposition 7.7.9]{Pedersen1979}.
\item If $H \subseteq G$ is a subgroup, then $\A \rtimes_{\alpha,r} H \subseteq \A \rtimes_{\alpha,r} G$.
\end{itemize}
We also remind the reader that if $\A$ is $C^*$-algebra and $\alpha \in \Aut(\A)$, then there exists a unique $\tilde{\alpha} \in \Aut(I(\A))$ such that $\tilde{\alpha}|_{\A}=\alpha$ \cite[Corollary 4.2]{Hamana1979}. Thus if $(\A,G,\alpha)$ is a discrete $C^*$-dynamical system, we have $\A \rtimes_{\alpha,r} G \subseteq I(\A) \rtimes_{\tilde{\alpha},r} G$.

\begin{theorem} \label{crossed product}
Let $(\A,G,\alpha)$ be a discrete $C^*$-dynamical system. If every pseudo-expectation for the $C^*$-inclusion $\A \subseteq \A \rtimes_{\alpha,r} G$ is faithful, then there exists a unique pseudo-expectation.
\end{theorem}

\begin{proof}
Suppose that every pseudo-expectation for $\A \subseteq \A \rtimes_{\alpha,r} G$ is faithful. To show that $\A \subseteq \A \rtimes_{\alpha,r} G$ has a unique pseudo-expectation, it suffices to show that $I(\A) \subseteq I(\A) \rtimes_{\tilde{\alpha},r} G$ has a unique pseudo-expectation, since every pseudo-expectation for the former inclusion extends to a pseudo-expectation for the latter inclusion. Indeed, if $\Theta:\A \rtimes_{\alpha,r} G \to I(\A)$ is a c.c.p. map such that $\Theta(a)=a$ for all $a \in \A$, then (by injectivity) there exists a c.c.p. $\tilde{\Theta}:I(\A) \rtimes_{\tilde{\alpha},r} G \to I(\A)$ such that $\tilde{\Theta}|_{\A \rtimes_{\alpha,r} G} = \Theta$. Then $\tilde{\Theta}|_{I(\A)}:I(\A) \to I(\A)$ is a c.c.p. map such that $\tilde{\Theta}(a)=a$ for all $a \in \A$. By the rigidity of the injective envelope, $\tilde{\Theta}(x)=x$ for all $x \in I(\A)$.\\

Now $I(\A)$ is an $AW^*$-algebra, in particular a unital $C^*$-algebra. Thus, by \cite[Theorem 3.5]{Zarikian2019a}, to show that $I(\A) \subseteq I(\A) \rtimes_{\tilde{\alpha},r} G$ has a unique pseudo-expectation, it suffices to show that $\tilde{\alpha}$ is properly outer. To that end, assume that there exists $e \neq g \in G$ and $0 \neq p \in \Proj(Z(I(\A)))$ such that $\tilde{\alpha}_g(p) = p$ and $\tilde{\alpha}_g|_{I(\A)p}$ is inner. Define $H = \langle g \rangle \subseteq G$. Since $I(\A)p$ is an $AW^*$-algebra \cite[Section 4, Proposition 8 (iii)]{Berberian1972}, Lemma \ref{inner action} implies there exists a unitary representation $u:H \to I(\A)p$ such that $\tilde{\alpha}_h|_{I(\A)p} = \Ad(u_h)$ for all $h \in H$. By the amenability of $H$ and the universal property of full crossed products \cite[Proposition 4.1.3]{BrownOzawa2008}, there exists a $*$-homomorphism $\pi:\A \rtimes_{\alpha,r} H \to I(\A)p$ such that $\pi(ah) = au_h$ for all $a \in \A$ and $h \in H$. In particular, $\pi(a) = ap$ for all $a \in \A$. Define a c.c.p. map $\theta:\A \rtimes_{\alpha,r} H \to I(\A)$ by the formula
\[
    \theta(x) = \pi(x)+\bbE(x)p^\perp, ~ x \in \A \rtimes_{\alpha,r} H,
\]
where $\bbE:\A \rtimes_{\alpha,r} H \to \A$ is the canonical faithful conditional expectation (see \cite[Proposition 4.1.9]{BrownOzawa2008}). Note that
\[
    \theta(a) = \pi(a) + \bbE(a)p^\perp = ap + ap^\perp = a, ~ a \in \A.
\]
Let $\Theta:\A \rtimes_{\alpha,r} G \to I(\A)$ be a c.c.p. extension of $\theta$. Then $\Theta$ is a pseudo-expectation for $\A \subseteq \A \rtimes_{\alpha,r} G$, therefore faithful. A fortiori, $\theta$ is faithful. Define $\J = \A \cap \A p$. Then $\J \neq \{0\}$ \cite[Lemma 1.2]{Hamana1982centre} and $I(\J) = I(\A)p$ \cite[Lemma 1.3]{Hamana1982centre}. Since $\J \subseteq \A$ is $\alpha$-invariant, $\J \rtimes_{\alpha,r} H \subseteq \A \rtimes_{\alpha,r} H$. Moreover, 
\[
    \theta(x) = \pi(x)+\bbE(x)p^\perp = \pi(x), ~ x \in \J \rtimes_{\alpha,r} H.
\]
Thus $\pi$ is faithful on $\J \rtimes_{\alpha,r} H$. By injectivity, there exists a c.c.p. map $\gamma:I(\J) \to I(\J)$ such that the following diagram commutes:
\begin{displaymath}
    \xymatrix{
        I(\J) \ar[drr]^{\gamma} & & \\
        \J \rtimes_{\alpha,r} H \ar[u]_{\pi} \ar[r]_{\bbE} & \J \ar[r] & I(\J)\\
    }
\end{displaymath}
For all $j \in \J$, we have that
\[
    \gamma(j) = \gamma(jp) = \gamma(\pi(j)) = \bbE(j) = j.
\]
By the rigidity of the injective envelope, we conclude that $\gamma=\id_{I(\J)}$. Thus $\pi=\bbE$. But then
\[
    ju_g = \pi(jg) = \bbE(jg) = 0, ~ j \in \J,
\]
which implies $\J = \{0\}$, a contradiction.
\end{proof}

We end this section with an example of a $C^*$-inclusion of the form $\A \subseteq \A \rtimes_{\alpha,r} G$ with a continuum of pseudo-expectations, all but two of which are faithful. This shows that one can come very close to satisfying the hypotheses of Theorem \ref{crossed product}, but fail the conclusion.

\begin{example} \label{almost all faithful}
Let $\A = K(\ell^2)+\bbC I \subseteq B(\ell^2)$ and $U \in B(\ell^2) \backslash \A$ be a self-adjoint unitary. Define $\alpha:\bbZ_2 \to \Aut(\A)$ by $\alpha_e=\id$ and $\alpha_g=\Ad(U)$. Then the pseudo-expectations for $\A \subseteq \A \rtimes_{\alpha,r} \bbZ_2$ are precisely the maps
\[
    \theta_t:\A \rtimes_{\alpha,r} \bbZ_2 \to B(\ell^2):A_0+A_1g \mapsto A_0+tA_1U, ~ -1 \leq t \leq 1.
\]
We have that $\theta_t$ is faithful for all $-1 < t < 1$, but $\theta_{\pm 1}$ are $*$-homomorphisms which are not faithful. Only $\theta_0$ is a conditional expectation.
\end{example}

\begin{proof}
Note that $I(\A) = B(\ell^2)$ \cite[Example 5.3]{Hamana1979opsys}. An easy computation shows that $\theta_{\pm 1}$ are $*$-homomorphisms, therefore c.c.p. maps. (Alternatively, this follows from the amenability of $\bbZ_2$ and the universal property of full crossed products.) Since $\theta_{\pm 1}$ restrict to the identity on $\A$, they are pseudo-expectations. Then $\theta_t$ is a pseudo-expectation for each $-1 < t < 1$, since it is a convex combination of $\theta_{\pm 1}$. Clearly $\theta_0=\bbE$, the canonical faithful conditional expectation. For $0 < t < 1$, we have that $\theta_t$ is faithful since it is a \ul{nontrivial} convex combination of $\theta_0$ and $\theta_1$. On the other hand, $\theta_1$ is not faithful, since for any $0 \neq T \in K(\ell^2)$, we have that $\theta_1(TU-Tg)=0$, even though $0 \neq TU-Tg \in \A \rtimes_{\alpha,r} \bbZ_2$. Similarly, $\theta_t$ is faithful for $-1 < t < 0$, while $\theta_{-1}$ is not faithful. If $t \neq 0$, then $\theta_t$ is not a conditional expectation, since $\theta_t(g) = tU \notin \A$.\\

It remains to show that there are no other pseudo-expectations. To that end, let $\theta:\A \rtimes_{\alpha,r} G \to B(\ell^2)$ be a c.c.p. map such that $\theta(A)=A$, for all $A \in \A$. Then
\[
    \theta(g)A = \alpha_g(A)\theta(g) = UAU\theta(g), ~ A \in \A,
\]
which implies $U\theta(g) \in \A' = \bbC I$. Thus $\theta(g) = tU$ for some $t \in \bbC$. Since $g^* = g^{-1} = g$ and $\theta$ is adjoint-preserving, we conclude that $t \in \bbR$. Since $\|g\| = 1$ and $\theta$ is contractive, we conclude that $-1 \leq t \leq 1$. Therefore $\theta = \theta_t$.
\end{proof}

\subsection{Twisted Crossed Products by Discrete Groups}

Let $(\A,G,\alpha,\sigma)$ be a discrete \ul{twisted} $C^*$-dynamical system \cite{PackerRaeburn1989}. In this section we show that Question \ref{Q1} has an affirmative answer for $C^*$-inclusions of the form $\A \subseteq \A \rtimes_{\alpha,r}^\sigma G$. We do this by using a ``stabilization'' trick to upgrade Theorem \ref{crossed product} from crossed products to twisted crossed products. Namely, we use that
\[
    \bbK_G \otimes (\A \rtimes_{\alpha,r}^\sigma G) \cong (\bbK_G \otimes \A) \rtimes_{\beta,r} G,
\]
for an (honest) action $\beta$ of $G$ on $\bbK_G \otimes \A$ \cite[Corollary 3.7]{PackerRaeburn1989}. (Here $\bbK_G$ denotes the compact operators on $\ell^2(G)$ and $\otimes$ is the minimal tensor product.) For our argument to work, we need to know that pseudo-expectations behave well with respect to stabilization, which in turn requires injective envelopes to behave well with respect to stabilization. Fortunately this is the case, as we now explain.\\

If $\A$ is a $C^*$-algebra and $J$ is any index set, we denote by $\bbM_J(\A)$ the set of all $J \times J$ matrices with entries in $\A$ whose finite submatrices are uniformly bounded in norm. This may not be a $C^*$-algebra, but it is an operator space (see \cite[Section 10.1]{EffrosRuan2000} or \cite[Section 1.2.26]{BlecherLeMerdy2004}). We denote by $\bbK_J(\A)$ the closure in $\bbM_J(\A)$ of the finitely-supported matrices, and note that this is a $C^*$-algebra with respect to the (well-defined) product
\[
    \begin{bmatrix} a_{ij} \end{bmatrix}\begin{bmatrix} b_{ij} \end{bmatrix} = \begin{bmatrix} \sum_{k \in J} a_{ik}b_{kj} \end{bmatrix}.
\]
Indeed, $\bbK_J(\A) \cong \bbK_J \otimes \A$. In fact, $\bbM_J(\A)$ is a contractive $\bbK_J(\A)$-bimodule with respect to the same product \cite[Section 1.2.27]{BlecherLeMerdy2004}. The following formula, which can be found in \cite[Corollary 4.6.12]{BlecherLeMerdy2004}, shows that the injective envelope behaves well with respect to stabilization:
\[
    I(\bbK_J(\A)) = \bbM_J(I(\A)).
\]

If $u:\A \to \B$ is a c.c.p. map between $C^*$-algebras, then the \emph{amplification}
\[
    u_J:\bbM_J(\A) \to \bbM_J(\B): \begin{bmatrix} a_{ij} \end{bmatrix} \mapsto \begin{bmatrix} u(a_{ij}) \end{bmatrix}
\]
is a completely contractive map between operator spaces. It restricts to a c.c.p. map $u_J:\bbK_J(\A) \to \bbK_J(\B)$, which can be identified with $\id \otimes u:\bbK_J \otimes \A \to \bbK_J \otimes \B$, where $\id:\bbK_J \to \bbK_J$ is the identity map.\\

In particular, if $\theta:\B \to I(\A)$ is a pseudo-expectation for the $C^*$-inclusion $\A \subseteq \B$, then the amplification $\theta_J:\bbK_J(\B) \to \bbK_J(I(\A))$ is a pseudo-expectation for the $C^*$-inclusion $\bbK_J(\A) \subseteq \bbK_J(\B)$, since
\[
    \bbK_J(I(\A)) \subseteq \bbM_J(I(\A)) = I(\bbK_J(\A)).
\]
The following theorem shows that if $\A \subseteq \B$ is an \ul{approximately unital} $C^*$-inclusion, then the pseudo-expectations for $\bbK_J(\A) \subseteq \bbK_J(\B)$ are \ul{precisely} the amplifications of the pseudo-expectations for $\A \subseteq \B$.

\begin{theorem} \label{amplification}
Let $\A \subseteq \B$ be an \ul{approximately unital} $C^*$-inclusion and $J$ be an index set. Then the map $\theta \mapsto \theta_J$ is a bijection between the pseudo-expectations $\theta:\B \to I(\A)$ for $\A \subseteq \B$ and the pseudo-expectations $\Theta:\bbK_J(\B) \to I(\bbK_J(\A))$ for $\bbK_J(\A) \subseteq \bbK_J(\B)$. Moreover, $\theta_J$ is faithful if and only if $\theta$ is faithful.
\end{theorem}

\begin{proof}
As noted in the preceding paragraph, if $\theta:\B \to I(\A)$ is a pseudo-expectation for $\A \subseteq \B$, then $\theta_J:\bbK_J(\B) \to \bbK_J(I(\A))$ is a pseudo-expectation for $\bbK_J(\A) \subseteq \bbK_J(\B)$.\\

Conversely, suppose $\Theta:\bbK_J(\B) \to I(\bbK_J(\A))$ is a pseudo-expectation for $\bbK_J(\A) \subseteq \bbK_J(\B)$. Let $\{e_\lambda\}$ be an approximate unit for $\A$ which is also an approximate unit for $\B$. For $i, j \in J$, denote by $E_{ij} \in \bbK_J$ the rank-one operator that maps $\delta_j \in \ell^2(J)$ to $\delta_i \in \ell^2(J)$. Fix $k \in J$ and define a pseudo-expectation $\theta:\B \to I(\A)$ for $\A \subseteq \B$ by the formula
\[
    \theta(b) = \Theta(E_{kk} \otimes b)_{kk}, ~ b \in \B.
\]
For all $i, j \in J$ and $b \in \B$,
\begin{eqnarray*}
    \Theta(E_{ij} \otimes b)
    &=& \Theta(\lim_\lambda\lim_\mu (E_{ij} \otimes e_\lambda be_\mu))\\
    &=& \lim_\lambda\lim_\mu \Theta(E_{ij} \otimes e_\lambda be_\mu)\\
    &=& \lim_\lambda\lim_\mu \Theta((E_{ik} \otimes e_\lambda)(E_{kk} \otimes b)(E_{kj} \otimes e_\mu))\\
    &=& \lim_\lambda\lim_\mu (E_{ik} \otimes e_\lambda)\Theta(E_{kk} \otimes b)(E_{kj} \otimes e_\mu)\\
    &=& \lim_\lambda\lim_\mu (E_{ij} \otimes e_\lambda\Theta(E_{kk} \otimes b)_{kk}e_\mu)\\
    &=& \lim_\lambda\lim_\mu (E_{ij} \otimes e_\lambda\theta(b)e_\mu)\\
    &=& E_{ij} \otimes \lim_\lambda\lim_\mu e_\lambda\theta(b)e_\mu.
\end{eqnarray*}
Denoting \emph{order limits} in $I(\A)$ by $\LIM$ (see \cite[Definition 2.1.20]{SaitoWright2015}), we have that
\[
    \LIM_\lambda e_\lambda = 1_{I(\A)}.
\]
Using \cite[Lemma 2.1.19 and Corollary 2.1.23]{SaitoWright2015}, it follows that
\[
    \lim_\lambda\lim_\mu e_\lambda\theta(b)e_\mu
    = \lim_\lambda\LIM_\mu e_\lambda\theta(b)e_\mu = \lim_\lambda e_\lambda\theta(b)\\
    = \LIM_\lambda e_\lambda\theta(b) = \theta(b).
\]
Thus for all $i, j \in J$ and $b \in \B$,
\[
    \Theta(E_{ij} \otimes b) = E_{ij} \otimes \theta(b) = \theta_J(E_{ij} \otimes b),
\]
which implies $\Theta=\theta_J$.\\

Now suppose that $\theta_J$ is faithful. If $\theta(b^*b)=0$, then
\[
    \theta_J((E_{ii} \otimes b)^*(E_{ii} \otimes b)) = \theta_J(E_{ii} \otimes b^*b) = E_{ii} \otimes \theta(b^*b) = 0,
\]
which implies $E_{ii} \otimes b = 0$, which in turn implies $b=0$. Conversely, suppose $\theta$ is faithful. If
\[
    \theta_J\left(\begin{bmatrix} b_{ij} \end{bmatrix}^*\begin{bmatrix} b_{ij} \end{bmatrix}\right)
    = \theta_J\left(\begin{bmatrix} b_{ji}^* \end{bmatrix}\begin{bmatrix} b_{ij} \end{bmatrix}\right) 
    = \begin{bmatrix} \theta\left(\sum_{k \in J} b_{ki}^*b_{kj}\right) \end{bmatrix} = 0,
\]
then $\sum_{k \in J} b_{ki}^*b_{ki}=0$ for all $i \in J$, which implies $\begin{bmatrix} b_{ij} \end{bmatrix}=0$.
\end{proof}

Here is the main result of this section.

\begin{theorem} \label{twisted}
Let $(\A,G,\alpha,\sigma)$ be a discrete twisted $C^*$-dynamical system. If every pseudo-expectation for the $C^*$-inclusion $\A \subseteq \A \rtimes_{\alpha,r}^\sigma G$ is faithful, then there exists a unique pseudo-expectation.
\end{theorem}

\begin{proof}
Note that the $C^*$-inclusion $\A \subseteq \A \rtimes_{\alpha,r}^\sigma G$ is approximately unital, since any approximate unit for $\A$ is also an approximate unit for $\A \rtimes_{\alpha,r}^\sigma G$. If every pseudo-expectation for $\A \subseteq \A \rtimes_{\alpha,r}^\sigma G$ is faithful, then every pseudo-expectation for $\bbK_G \otimes \A \subseteq \bbK_G \otimes (\A \rtimes_{\alpha,r}^\sigma G)$ is faithful, by Theorem \ref{amplification}. By \cite[Corollary 3.7]{PackerRaeburn1989}, there exists an (honest) action $\beta$ of $G$ on $\bbK_G \otimes \A$ such that
\[
    \bbK_G \otimes (\A \rtimes_{\alpha,r}^\sigma G) \cong (\bbK_G \otimes \A) \rtimes_{\beta,r} G
\]
via an isomorphism which identifies the copies of $\bbK_G \otimes \A$. Thus every pseudo-expectation for $\bbK_G \otimes \A \subseteq (\bbK_G \otimes \A) \rtimes_{\beta,r} G$ is faithful, and so there is a unique pseudo-expectation by Theorem \ref{crossed product}. Then $\bbK_G \otimes \A \subseteq \bbK_G \otimes (\A \rtimes_{\alpha,r}^\sigma G)$ has a unique pseudo-expectation, which implies $\A \subseteq \A \rtimes_{\alpha,r}^\sigma G$ has a unique pseudo-expectation (again by Theorem \ref{amplification}).
\end{proof}

Another trick allows one to upgrade from inclusions of the form $\A \subseteq \A \rtimes_{\alpha,r}^\sigma G$ to inclusions of the form $\A \rtimes_{\alpha,r}^\sigma N \subseteq \A \rtimes_{\alpha,r}^\sigma G$, for a normal subgroup $N \lhd G$.

\begin{corollary} \label{normal subgroup}
Let $(\A,G,\alpha,\sigma)$ be a discrete twisted $C^*$-dynamical system and $N \lhd G$ be a normal subgroup. If every pseudo-expectation for the $C^*$-inclusion $\A \rtimes_{\alpha,r}^\sigma N \subseteq \A \rtimes_{\alpha,r}^\sigma G$ is faithful, then there exists a unique pseudo-expectation.
\end{corollary}

\begin{proof}
By \cite[Theorem 4.1]{PackerRaeburn1989}, there exists a twisted $C^*$-dynamical system $(\A \rtimes_{\alpha,r}^\sigma N,G/N,\beta,\tau)$ such that
\[
    \A \rtimes_{\alpha,r}^\sigma G \cong (\A \rtimes_{\alpha,r}^\sigma N) \rtimes_{\beta,r}^\tau (G/N)
\]
via an isomorphism which identifies the copies of $\A \rtimes_{\alpha,r}^\sigma N$. Then the result follows immediately from Theorem \ref{twisted}.
\end{proof}

\subsection{Subalgebras of $B(\H)$}

In this section, which is unrelated to the previous two, we show that Question \ref{Q1} has an affirmative answer for unital $C^*$-inclusions $\A \subseteq B(\H)$.

\begin{theorem} \label{B(H)}
If every pseudo-expectation for the unital $C^*$-inclusion $\A \subseteq B(\H)$ is faithful, then there exists a unique pseudo-expectation. In that case, we can choose $I(\A) = \A''$ and the unique pseudo-expectation is a faithful normal conditional expectation $E:B(\H) \to \A''$.
\end{theorem}

\begin{proof}
By \cite[Theorem 3.12]{PittsZarikian2015}, $\A'$ is abelian, therefore injective. It follows that $\A''$ is injective \cite[Theorem IV.2.2.7]{Blackadar2006}, which allows us to choose $I(\A) \subseteq \A''$. We will now show that the reverse inclusion holds as well. To that end, let $\theta:B(\H) \to I(\A)$ be a pseudo-expectation for $\A \subseteq B(\H)$ (faithful, by assumption). Then $\theta|_{I(\A)}:I(\A) \to I(\A)$ is a c.c.p. map such that $\theta|_{\A} = \id_{\A}$. By the rigidity of the injective envelope, $\theta|_{I(\A)} = \id_{I(\A)}$ (see \cite[Theorem 6.2.1]{EffrosRuan2000} or \cite[Lemma 4.2.4]{BlecherLeMerdy2004}). If $\{x_i\} \subseteq I(\A)_{sa}$ is an increasing net with supremum $x \in B(\H)_{sa}$, then $x_i = \theta(x_i) \leq \theta(x)$ for all $i$, which implies $x \leq \theta(x)$. Since $\theta(x) - x \geq 0$ and $\theta(\theta(x) - x) = 0$, we conclude that $x = \theta(x) \in I(\A)_{sa}$. Likewise if $\{x_i\} \subseteq I(\A)_{sa}$ is a decreasing net with infimum $x \in B(\H)_{sa}$, then $x \in I(\A)_{sa}$. By Pedersen's Up-Down-Up Theorem \cite[Theorem II.4.24]{Takesaki1979},
\[
    (\A'')_{sa} = (((\A_{sa})^m)_m)^m \subseteq I(\A)_{sa},
\]
which implies $\A'' \subseteq I(\A)$.\\

Since $I(\A) = \A''$, every conditional expectation for the inclusion $\A'' \subseteq B(\H)$ is a pseudo-expectation for the inclusion $\A \subseteq B(\H)$. Conversely, every pseudo-expectation for $\A \subseteq B(\H)$ is a conditional expectation for $\A'' \subseteq B(\H)$, by the rigidity of the injective envelope. Thus every conditional expectation for $\A'' \subseteq B(\H)$ is faithful. By \cite[Proposition 7.5]{PittsZarikian2015}, there exists a unique conditional expectation for $\A'' \subseteq B(\H)$, which is faithful and normal, and this is the unique pseudo-expectation for $\A \subseteq B(\H)$.
\end{proof}

\section{Negative Results}

\subsection{A Counterexample}

In this section, we show that in general Question \ref{Q1} has a negative answer. We thank Mikael R{\slasho}rdam for suggesting that we search for counterexamples among $C^*$-irreducible inclusions of the form $C_r^*(G) \subseteq C(X) \rtimes_{\alpha,r} G$.

\begin{theorem} \label{counterexample}
There exist $C^*$-irreducible inclusions of the form $C_r^*(G) \subseteq C(X) \rtimes_r G$ with multiple conditional expectations.
\end{theorem}

\begin{proof}
Let $G$ be a countably-infinite discrete group such that $C_r^*(G)$ is simple (one could take $G=\bbF_2$, by a result of Powers \cite[Theorem 2]{Powers1975}). By \cite[Corollary 1.1]{Elek2021}, there exists a free minimal action $G \curvearrowright X$ on the Cantor set that is not uniquely ergodic. Denote by $\alpha$ the induced action $G \curvearrowright C(X)$. By \cite[Theorem 1.3]{AmrutamKalantar2020}, the inclusion $C_r^*(G) \subseteq C(X) \rtimes_{\alpha,r} G$ is $C^*$-irreducible. But if $\mu$ is any $G$-invariant Borel probability measure on $X$ and $\phi_\mu$ is the $G$-invariant state on $C(X)$ defined by integration against $\mu$, then there exists a unique conditional expectation $E_\mu:C(X) \rtimes_{\alpha,r} G \to C_r^*(G)$ such that
\[
    E_\mu\left(\sum_g a_gg\right) = \sum_g \phi_\mu(a_g)g, ~ a \in C_c(G,C(X)).
\]
(See \cite[Exercise 4.1.4]{BrownOzawa2008}.) In particular,
\[
    E_\mu(a) = \phi_\mu(a) = \int_X a(x)d\mu(x), ~ a \in C(X).
\]
Since there are multiple $G$-invariant Borel probability measures on $X$, there are multiple conditional expectations for $C_r^*(G) \subseteq C(X) \rtimes_{\alpha,r} G$. 
\end{proof}

\subsection{The Main Question Reformulated}

A principal difference between inclusions of the form $\A \subseteq \A \rtimes_{\alpha,r}^\sigma G$, for which Question \ref{Q1} has an affirmative answer by Theorem \ref{twisted}, and inclusions of the form $C_r^*(G) \subseteq C(X) \rtimes_{\alpha,r} G$, for which Question \ref{Q1} can have a negative answer by Theorem \ref{counterexample}, is that the former inclusions are regular while the latter inclusions need not be. The $C^*$-inclusion $\A \subseteq \B$ is \emph{regular} if $\A$ has an abundance of \emph{normalizers} in $\B$, in the following sense:
\[
    \B = \o{\span}\{n \in \B: n\A n^* \cup n^*\A n \subseteq \A\}.
\]
This suggests that Question \ref{Q1} should be modified into the following:

\begin{question} \label{Q3}
If every pseudo-expectation for a \ul{regular} $C^*$-inclusion $\A \subseteq \B$ is faithful (equivalently, if $\A \subseteq \B$ is regular and hereditarily essential), must there be a unique pseudo-expectation?
\end{question}

In conclusion, we highlight results from the literature which show that Question \ref{Q3} has an affirmative answer provided $\A$ is either (I) abelian or (II) simple.
\begin{enumerate}
\item[(I)] Suppose $\A$ is abelian and $\A \subseteq \B$ is regular. Pitts proves that ``restricting the domain'' yields a bijection between the pseudo-expectations for $\A \subseteq \B$ and the pseudo-expectations for $\A \subseteq \A' \cap \B$ \cite[Proposition 6.10]{Pitts2021}. If $\A \subseteq \B$ is hereditarily essential, then (by definition) $\A \subseteq \A' \cap \B$ is essential and (by \cite[Theorem 3.12]{PittsZarikian2015}) $\A' \cap \B$ is abelian. It follows that $\A \subseteq \A' \cap \B$ has a unique pseudo-expectation, by \cite[Corollary 3.22]{PittsZarikian2015}. Thus $\A \subseteq \B$ has a unique pseudo-expectation, by the aforementioned result of Pitts.
\item[(II)] Suppose $\A$ is simple and $\A \subseteq \B$ is regular. Kwa\'{s}niewski proves that if $\A \subseteq \B$ is hereditarily essential, then it has a unique pseudo-expectation (which is actually a conditional expectation!).\footnote{To the author's knowledge, this result has not been published yet. But it has been presented in considerable detail in the Prague Noncommutative Geometry and Topology Seminar (May 13, 2022) \cite{KwasniewskiLecture2022} and the Workshop on Ideal Structure of $C^*$-Algebras from Dynamics and Groups (April 8-12, 2024) \cite{KwasniewskiLecture2024}.}
\end{enumerate}

\begin{remark}
Concerning (II) above, Kwa\'{s}niewski actually proves that for regular inclusions $\A \subseteq \B$ with $\A$ simple, being hereditarily essential has a \ul{host} of equivalent reformulations. One of those, as mentioned in (II) above, is having a unique pseudo-expectation which is faithful. Another is being \emph{aperiodic} and having a faithful pseudo-expectation (see Section \ref{crossed} above for a reminder about aperiodicity). So a more ambitious (but still reasonable) goal than answering Question \ref{Q3} affirmatively would be to prove the equivalence of the following three conditions for \ul{regular inclusions} $\A \subseteq \B$ \ul{with a faithful pseudo-expectation}:
\begin{enumerate}
\item[i.] $\A \subseteq \B$ is hereditarily essential.
\item[ii.] $\A \subseteq \B$ is aperiodic.
\item[iii.] $\A \subseteq \B$ has a unique pseudo-expectation.
\end{enumerate}
The implication (ii $\implies$ iii) and the equivalence (iii $\iff$ i) are known, by \cite[Theorem 3.6]{KwasniewskiMeyer2022} and \cite[Theorem ]{PittsZarikian2015}, respectively. So replacing Question \ref{Q3} by the equivalence of all three conditions amounts to trying to prove (i $\implies$ ii) instead of (i $\implies$ iii).
\end{remark}

\section*{Acknowledgements}

The author would like to express his gratitude to the anonymous referees for numerous helpful suggestions that improved the results, clarity, and attributional accuracy of this paper.

\end{document}